\documentclass[11pt,oneside]{article}
\usepackage[english]{babel}

\usepackage[dvips]{graphicx}
\usepackage[svgnames]{xcolor}
\colorlet{foldline}{Red}
\colorlet{vector}{RoyalBlue!100!black}
\colorlet{halfgray}{darkgray}

\usepackage{amssymb}
\usepackage[tbtags]{amsmath}
\usepackage{amsthm}
\usepackage{icomma}
\usepackage{units}

\theoremstyle{plain}
\newtheorem{thm}{Theorem}
\newtheorem{lem}[thm]{Lemma}

\theoremstyle{definition}
\newtheorem{exmp}{Example}

\usepackage[T1]{fontenc}
\usepackage[osf,sc]{mathpazo}
\usepackage[euler-digits,euler-hat-accent]{eulervm}

\usepackage[top=32mm]{geometry}

\usepackage{booktabs}
\usepackage[flushleft]{threeparttable}
\usepackage{tabularx}

\usepackage{caption}
\usepackage[authoryear,round,longnamesfirst]{natbib}
\usepackage{fancyhdr}
\usepackage{hyperref}
\hypersetup{
  breaklinks = true,
  linktocpage=true,
  linkcolor  = MidnightBlue,
  citecolor  = MidnightBlue,
  urlcolor   = MidnightBlue,
  colorlinks = true,
}
\usepackage{breakurl}
\usepackage{breakcites}    

\usepackage{pst-all} 
\newpsobject{showgrid}{psgrid}{subgriddiv=1,griddots=10,gridlabels=6pt}
\psset{arrowscale=2}

\title{Division of an angle into equal parts and construction of regular polygons by multi-fold origami}
\author{Jorge C. Lucero\thanks{Dept.\ Computer Science, University of Bras\'{i}lia, Brazil. E-mail: lucero@unb.br}} 
\date{\today} 

\begin{document}

\maketitle
\thispagestyle{empty}

\begin{abstract} 
\noindent 
This article analyses geometric constructions by origami when up to $n$ simultaneous folds may be done at each step. It shows that any arbitrary angle can be $m$-sected if the largest prime factor of $m$ is $p\le n+2$. Also, the regular $m$-gon can be constructed if the largest prime factor of $\phi(m)$ is $q\le n+2$, where $\phi$ is Euler's totient function. 
  
\end{abstract}

\section{Introduction}

Two classic construction problems of plane geometry are the division of an arbitrary angle into equal parts and the construction of regular polygons \citep{Martin1998}. 
It is well known that the use of straight edge and compass allows for the bisection of angles and the constructions of regular $m$-gons if $m=2^ap_1p_2\cdots p_k$, where $a, k\ge 0$ and each $p_i$ is a distinct odd prime of the form $p_i=2^b_i+1$. It is also known that origami extends the constructions by allowing for the trisection of angles and the constructions of regular $m$-gons if $m=2^{a_1}3^{a_2} p_1p_2\cdots p_k$, where $a_1, a_2, k\ge 0$ and each $p_i$ is a distinct prime of the form $p_i=2^{b_{i,1}}3^{b_{i,2}}+1>3$ \citep{Alperin2000}. 

Standard origami constructions are performed by a sequence of elementary single-fold operations, one at a time. Each elementary operation solves a set of specific incidences constraints between given points and lines and their folded images \citep{Alperin2000, Alperin2006,Justin1986}. A total of eight elementary operations may be defined and stated as in Table \ref{axioms} \citep{Lucero2017}. 
The operations can solve arbitrary cubic equations \citep{Geretschlager1995, Hull2011}, and therefore they can be applied to related construction problems such as the duplication of the cube \citep{Messer1986} and those mentioned above \citep{Geretschlager1995,Geretschlager1997,Geretschlager1997b}.

\begin{table}
\centering
\caption{Single-fold operations \citep{Lucero2017}. $\mathcal{O}$ denotes the medium in which folds are performed; e.g., a sheet of paper, fabric, plastic, metal or any other foldable material.}
\label{axioms}

\begin{tabularx}{\textwidth}{cX}
\hline
\# &\multicolumn{1}{c}{Operation}\\
\hline
1 & Given two distinct points $P$ and $Q$, fold $\mathcal{O}$ to place $P$ onto $Q$.\\
2 &  Given two distinct lines $r$ and $s$, fold $\mathcal{O}$ to align $r$ and $s$.\\
3 & Fold along a given a line $r$.\\
4 & Given two distinct points $P$ and $Q$, fold $\mathcal{O}$ along a line passing through $P$ and $Q$.\\
5 & Given a line $r$ and a point $P$, fold $\mathcal{O}$ along a line passing through $P$ to reflect $r$ onto itself.\\
6 & Given a line $r$, a point $P$ not on $r$ and a point $Q$, fold $\mathcal{O}$ along a line passing through $Q$ to place $P$ onto $r$.\\
7 & Given two lines $r$ and $s$, a point $P$ not on $r$ and a point $Q$ not on $s$, where $r$ and $s$ are distinct or $P$ and $Q$ are distinct, fold $\mathcal{O}$ to place $P$ onto $r$, and $Q$ onto $s$.\\
8 & Given two lines $r$ and $s$, and a point $P$ not on $r$, fold $\mathcal{O}$ to place $P$ onto $r$, and to reflect $s$ onto itself.\\
\hline
\end{tabularx}
\end{table}

The range of origami constructions may be extended further by using multi-fold operations, in which up to $n$ simultaneous folds may be performed at each step \citep{Alperin2006}, instead of single folds. In the case of $n=2$, the set of possible elementary operations increases to  209 or more (the exact number has still not been determined). It has been shown that 2-fold origami allows for the geometric solution of arbitrary septic equations \citep{Konig2016}, quintisection of an angle \citep{Lang2004a} and construction of the regular hendecagon \citep{Lucero2018}.

Thus, the purpose of this article is to analyze the general case of $n$-fold origami with arbitrary $n\ge1$ and determine what angle divisions and regular polygons can be obtained.

\section{Single- and multi-fold origami}

An $n$-fold elementary operation is the resolution of a minimal set of incidence constraints between given points, lines, and their folded images, that defines a finite number of sets of $n$ fold lines \citep{Alperin2006}. 
For the case of $n=1$, all possible elementary operations are those listed in Table \ref{axioms}. 
An example of operation for $n=2$ is illustrated in Fig.~\ref{nishi}. 

\begin{figure}
\centering
\begin{pspicture}(-3,-2.5)(3.5,2.5)
\psset{xunit=.6cm,yunit=.6cm}
\rput{-20}{%
\psplot[linewidth=1pt]{-3.7}{-2}{x 3 add 3 mul}
\uput[180]{*0}(-2.3,2.2){$\ell$}
\psline[linewidth=1pt,linecolor=foldline](0,-2)(0,3)
\uput[0]{*0}(0,2.5){$\delta$}
\psplot[linewidth=1pt,linecolor=foldline]{2}{3.7}{x -3 add -3 mul}
\uput[0]{*0}(2,2.7){$\gamma$}
\psplot[linewidth=1pt]{-1}{2.4}{x 1 add -1.5 mul 3 add}
\uput[180]{*0}(-.5,2.2){$s$}
\qdisk(-1.67,-1){2pt}
\uput[90]{*0}(-1.67,-1){$Q$}
\qdisk(1.67,-1){2pt}
\psplot[linewidth=1pt]{3.29}{3.8}{x -3.5 add 10 mul}
\uput[0]{*0}(3.8,2.7){$r$}
\qdisk(3.7,2){2pt}
\qdisk(1.24,1.18){2pt}
\uput[90]{*0}(1.24,1.18){$P$}
\psarcn[linecolor=halfgray]{->}(3.3,-1){1.8}{135}{84}
\psarc[linecolor=halfgray]{->}(0,3){2.6}{248}{292}
\psarc[linecolor=halfgray]{->}(0,4){2.6}{230}{310}
}
\end{pspicture}
\caption{A two-fold operation \cite{Lucero2018}. Given two points $P$ and $Q$ and three lines $\ell$, $r$, $s$, simultaneously fold along a line $\gamma$ to place $P$ onto $r$, and along a line $\delta$ to place $Q$ onto $s$ and to align $\ell$ and $\gamma$.} 
\label{nishi}
\end{figure}
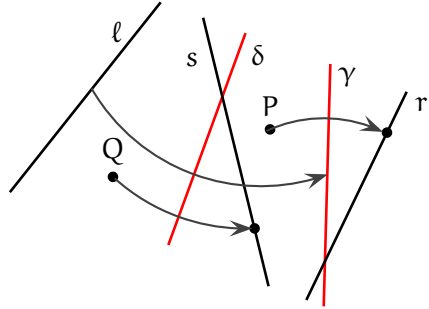

Any number of $n_i$-fold operations, $i=1, 2, \ldots, k$, may be gather together and considered as a unique $n$-fold operation, with $n=\sum_{i=1}^k n_i$. 
Thus,
we define $n$-fold origami as the construction tool consisting of all the $k$-fold elementary operations, with $1\le k\le n$.

The medium on which all folds are performed is assumed to be an infinite Euclidean plane. Points are referred by their Cartesian $xy$-coordinates or by identifying them as complex numbers, as convenient. A point or complex number is said to be \textit{$n$-fold constructible} iff it can be constructed starting from numbers 0 and 1 and applying a sequence of $n$-fold operations.
It has been shown that the set of constructible numbers in $\mathbb{C}$ by single-fold origami is the smallest subfield of $\mathbb{C}$ that is closed under square roots, cube roots and complex
conjugation \citep{Alperin2000}. An immediate corollary is that the field $\mathbb{Q}$ of rational numbers is $n$-fold  constructible, for any $n\ge 1$.

The present analysis is based on the following version of a theorem on polynomial root construction \citep{Alperin2006}.

\begin{thm}
The real roots of any $m$th-degree polynomial with $n$-fold constructible coefficients are $n$-fold constructible if $m\le n+2$.
\label{thm_a} 
\end{thm}

\begin{proof}
The real roots of any $m$th-degree polynomial may be obtained by \citeauthor{Lill1867}'s \citeyearpar{Lill1867} method  \citep[see also][]{Hull2011,Riaz1962}. It consists of defining first a right-angle path from and origin $O$ to a terminus $T$, where the lengths and directions of the path's segments are given by the non-zero coefficients of the polynomial. Next, a second right-angle
path with $m$ segments between $O$ and $T$ is constructed by folding, and this construction demands the execution of $m-2$ simultaneous folds, if $m\ge 3$, or a single fold, if $m\le 3$. The first intersection (from $O$) between both paths is the sought solution. 

Details of the method may be found in the cited references.  An example for solving $x^5-a=0$ is shown in Fig.~\ref{mroot}.
\end{proof}

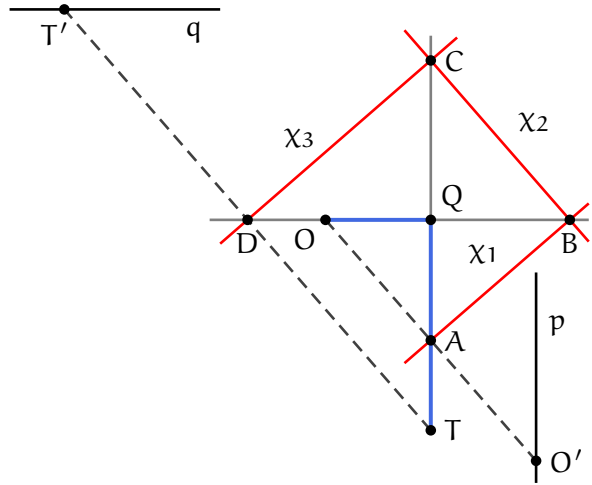
\begin{figure}
\centering
\psset{xunit=.7cm,yunit=.7cm}
\begin{pspicture}(-7,-5)(5.5,4)
\psline[linewidth=1pt](4,-1)(4,-5)
\psline[linewidth=1pt](-6,4)(-2,4)
\psline[linecolor=gray,linewidth=1pt](2,0)(2,3.5)
\psline[linecolor=gray,linewidth=1pt](-2.2,0)(5,0)

\psline[linecolor=vector,linewidth=1.5pt](0,0)(2,0)
\psline[linecolor=vector,linewidth=1.5pt](2,0)(2,-4)

\psline[linecolor=halfgray,linewidth=1pt, linestyle=dashed](0,0)(4,-4.58)
\psline[linecolor=halfgray,linewidth=1pt, linestyle=dashed](2,-4)(-4.96,4)

\psplot[linewidth=1pt,linecolor=foldline,plotpoints=10]{1.5}{5}{x -2 add 2.29 mul 2.64 div -2.29 add}
\psplot[linewidth=1pt,linecolor=foldline,plotpoints=10]{-2}{2.5}{x 1.48 add 3.03 mul 3.48 div}
\psplot[linewidth=1pt,linecolor=foldline,plotpoints=10]{1.5}{5}{x -4.64 add 3.03 mul -2.64 div}

\qdisk(0,0){2pt}
\uput[-135](0,0){$O$}
\qdisk(2,0){2pt}
\uput[45](2,0){$Q$}
\qdisk(2,-4){2pt}
\uput[0](2,-4){$T$}

\qdisk(2,-2.29){2pt}
\uput[0](2,-2.29){$A$}
\qdisk(4.64,0){2pt}
\uput[-90](4.64,0){$B$}
\qdisk(2,3.03){2pt}
\uput[0](2,3.03){$C$}
\qdisk(-1.48,0){2pt}
\uput[-90](-1.48,0){$D$}

\qdisk(4,-4.58){2pt}
\uput[0](4,-4.58){$O'$}
\qdisk(-4.96,4){2pt}
\uput[-110](-4.96,4){$T'$}

\uput[0](4,-2){$p$}
\uput[-90](-2.5,4){$q$}

\uput[135](3.5,-1){$\chi_1$}
\uput[45](3.5,1.5){$\chi_2$}
\uput[135](0,1.2){$\chi_3$}

\end{pspicture}
\caption{Geometrical solution of $x^5-a=0$ by 3-fold origami. Set perpendicular segments $\overline{OQ}$ and $\overline{QT}$ with respective lengths 1 and $a$, line $p$ parallel to $\overline{QT}$ at a distance of 1, and line $q$ parallel to $\overline{OQ}$ at a distance of $a$. Next, construct Lill's path $\overline{OA}$, $\overline{AB}$, $\overline{BC}$, $\overline{CD}$, $\overline{DT}$ by performing three simultaneous folds: fold $\chi_1$ places  point $O$ onto line $p$, fold $\chi_2$ is perpendicular to $\chi_1$ and passes through the intersection of $\chi_1$ with the direction line of  $\overline{OQ}$ (point $B$), and fold $\chi_3$ is perpendicular to $\chi_2$, passes through the intersection of $\chi_2$ with the direction line of  $\overline{QT}$ (point $C$), and places  point $T$ onto line $q$. Point $A$ is at the intersection of $\chi_1$ with the direction line of $QT$, and the length of $\overline{QA}$ is $\sqrt[5]{a}$.}  
\label{mroot}
\end{figure}

It must be noted that the roots of 5th- and 7th-degree polynomials may be obtained by 2-fold origami, instead of the 3- and 5-fold origami, respectively, predicted by the above theorem \citep{Nishimura2015,Konig2016}. Therefore, Theorem \ref{thm_a} only posses a sufficient condition on the number of simultaneous folds required.  

\section{Angle section}

Let us consider first the case of division into any prime number of parts.

\begin{lem}
Any angle may be divided into $p$ equal parts by $n$-fold origami if $p$ is a prime and $p\le n+2$.
\label{thm_angle}
\end{lem}

\begin{proof}
Let $\ell$ be a line forming an angle $\theta$ with the $x$-axis on the plane. Then, point $P(\cos \theta, 0)$ may be constructed as shown in Fig.~\ref{angle}. 

Consider next the multiple angle identity
\begin{equation}
\cos(p\alpha)=T_p(\cos \alpha)
\label{cheb}
\end{equation}
where $T_p$ is the $p$th Chebyshev polynomial of the first kind, defined by
\begin{align}
T_0(x) & =1,\\
T_1(x) & = x,\\
T_{p+1}(x) & = 2xT_p(x)-T_{p-1}(x).
\end{align}

\begin{figure}
\centering
\begin{pspicture}(-1,-1.)(5,5)
\psset{xunit=1cm,yunit=1cm}
\psaxes[linecolor=gray,labels=none,ticks=none,linewidth=1pt]{->}(0,0)(-.5,-.5)(5,5)[$x$,0][$y$,0]
\psarc[linecolor=gray,linewidth=1pt](0,0){4}{-10}{100}
\psarc[linecolor=gray,linewidth=1pt](0,0){1}{0}{45}
\uput[5]{*0}(.9,0.2){$\theta$}
\psline[linewidth=1pt,plotpoints=10](0,0)(4, 4)
\psline[linecolor=foldline,linewidth=1pt,plotpoints=10](0,0)(5.22,2.16)
\psline[linecolor=foldline,linewidth=1pt,plotpoints=10](2.83,-0.5)(2.83,3.5)
\psline[linewidth=1pt,linestyle=dashed,plotpoints=10](4,0)(2.83, 2.83)
\uput[90]{*0}(2.8,3.5){$\chi_2$}
\uput[135]{*0}(3.8,3.8){$\ell$}
\uput[110]{*0}(5,2){$\chi_1$}
\uput[-45]{*0}(2.83,0){$P$}
\uput[-45]{*0}(4,0){$Q$}
\uput[0]{*0}(2.83, 2.83){$Q'$}
\uput[-45]{*0}(0,0){$O$}
\qdisk(2.83, 2.83){2pt}
\qdisk(2.83,0){2pt}
\qdisk(4,0){2pt}
\qdisk(0,0){2pt}
\end{pspicture}
\caption{Construction for Lemma \ref{thm_angle}. Given points $O(0, 0)$, $Q(1, 0)$, and line $\ell$ forming an angle $\theta$ with  $\overline{OQ}$: (1) fold along a line ($\chi_1$) to place $\ell$ onto  $\overline{OQ}$, and next (2) fold along a  perpendicular ($\chi_2$) to $\overline{OQ}$ passing through $Q'$. The intersection of $\overline{OQ}$ and $\chi_2$ is $P=(\cos \theta$, 0).} 
\label{angle}
\end{figure}
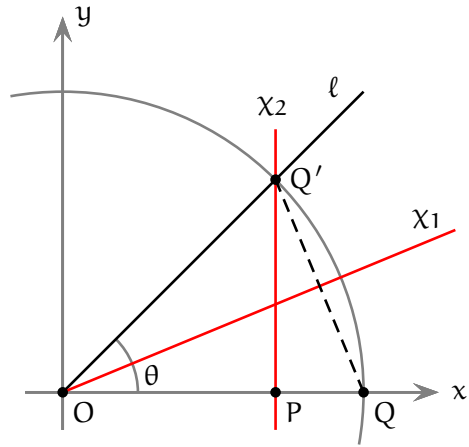

Letting $\theta=p\alpha$, then Eq.~(\ref{cheb}) is a $p$th-degree polynomial equation on $x=\cos(\theta/p)$ with integer (constructible) coefficients. According to Theorem \ref{thm_a}, the equation may be solved by $(p-2)$-fold origami, if $p\ge 3$, or single-fold origami, if $p\le 3$. Then, a line $\ell'$ forming an angle $\theta/p$ may be constructed from $\cos(\theta/p)$ by reversing the procedure in Fig.~\ref{angle}. 
\end{proof}

The lemma is easily extended to the general case of division into an arbitrary number of parts. 

\begin{thm}
Any angle may be divided into $m\ge 2$ equal parts by $n$-fold origami if 
the largest prime factor $p$ of $m$ satisfies $p\le n+2$.
\label{thm_angle2}
\end{thm}

\begin{proof}
Let $m=p_1p_2\cdots p_k$, where each $p_i$ is a prime and $p_i\le n+2$. Then, the theorem is proved by induction over $k$ and applying Lemma \ref{thm_angle}.
\end{proof}

Again, we remark that the above theorem only posses a sufficient condition on the number of multiple folds required. For 
$m=5$, it predicts $n=3$; however, a solution using only 2-fold origami has been published \citep{Lang2004a}.

\begin{exmp}
Any angle may be divided into 11 equal parts by 9-fold origami.
\end{exmp}

\section{Regular polygons}

The analysis follows similar steps to previous treatments on geometric constructions by single-fold origami and other tools \citep{Gleason1988,Stewart2015,Videla1997}.

Consider an $m$-gon ($m\ge 3)$ circumscribed in a circle with radius 1 and centered at the origin in the complex plane. Its vertices are given by the $m$th-roots of unity, which are the solutions of
$z^m-1=0$. 

Let us recall that an $m$th root of unity is primitive if it is not a $k$th root of unity for $k<m$. The primitive $m$th roots are solutions of the $m$th cyclotomic polynomial
\begin{equation}
\Phi_m(z)=\prod_{\stackrel{1\le k \le m}{\gcd(k, m) =1}} \left( z-e^{\nicefrac{2i\pi k}{m}}\right).
\end{equation}
This polynomial has degree $\phi(m)$, where $\phi$ is Euler's totient function; i.e., $\phi(m)$ is  the number of positive integers $k\le m$ that are coprime to $m$.
A property of any $m$th primitive root $\xi_m$ is that all the $m$ distinct roots may be obtained as $\xi_m^k$, for $k=0, 1, \ldots, m-1$. This property provides a convenient way to construct the regular $m$-gon. 
 
\begin{lem}
The regular $m$-gon is $n$-fold constructible if a primitive $m$th root of unity is $n$-fold constructible.
\label{lemmaroot}
\end{lem}

\begin{proof}
 Let $\xi_m=e^{i\theta}$ be a primitive $m$th root of unity. Then, $\xi_m^k=e^{ik\theta}$ and therefore all roots may be constructed from $\xi_m$ by applying rotations of an angle $\theta$ around the origin. The rotations may be performed by single-fold origami, as shown in Fig. \ref{rotation}. Once all the roots have been constructed, segments connecting consecutive roots may be created by single folds.  
\end{proof}

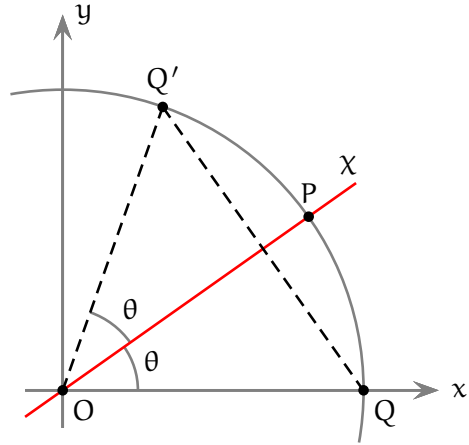
\begin{figure}
\centering
\begin{pspicture}(-1,-1.)(5,5)
\psset{xunit=1cm,yunit=1cm}
\psaxes[linecolor=gray,labels=none,ticks=none,linewidth=1pt]{->}(0,0)(-.5,-.5)(5,5)[$x$,0][$y$,0]
\psarc[linecolor=gray,linewidth=1pt](0,0){4}{-10}{100}
\psarc[linecolor=gray,linewidth=1pt](0,0){1}{0}{35.3}
\psarc[linecolor=gray,linewidth=1pt](0,0){1.1}{35.3}{70.6}
\uput[17]{*0}(.9,0.3){$\theta$}
\uput[52]{*0}(0.7,.8){$\theta$}
\psplot[linecolor=foldline,linewidth=1pt,plotpoints=10]{-.5}{3.9}{x .707 mul}
\psplot[linewidth=1pt,linestyle=dashed,plotpoints=10]{1.33}{4}{x -1.414 mul 5.65 add}
\psline[linewidth=1pt,linestyle=dashed,plotpoints=10](0,0)(1.33,3.77)
\uput[90]{*0}(3.8,2.7){$\chi$}
\uput[90]{*0}(1.33,3.77){$Q'$}
\uput[-45]{*0}(4,0){$Q$}
\uput[90]{*0}(3.27,2.31){$P$}
\uput[-45]{*0}(0,0){$O$}
\qdisk(3.27,2.31){2pt}
\qdisk(1.33,3.77){2pt}
\qdisk(4,0){2pt}
\qdisk(0,0){2pt}
\end{pspicture}
\caption{Given $O = (0, 0)$, $Q = (1, 0)$ and $P=(\cos \theta, \sin \theta)$, a fold along line $\chi$ passing through $O$ and $Q$ places $Q$ on $Q'=(\cos 2\theta, \sin 2\theta)$.} 
\label{rotation}
\end{figure}

Next,  
we state
a sufficient condition for the $n$-fold constructability of a number $\alpha \in \mathbb{C}$.
\begin{lem}
A number $\alpha \in \mathbb{C}$ is $n$-fold constructible if there is a field tower  $\mathbb{Q}=F_0\subseteq F1 \subseteq \cdots \subseteq F_{k-1} \subseteq F_k \subset \mathbb{C}$, such that $\alpha\in F_k$ and $[F_j : F_{j-1}]\in\{2, 3, \ldots, n+2\}$for each $j=1, 2, \ldots, k$.
\label{alpha}
\end{lem}

\begin{proof}
The theorem is proved by induction over $k$. If $k=0$, then $\alpha\in F_0=\mathbb{Q}$ is constructible by single-fold origami \citep{Alperin2000}, and therefore is $n$-fold constructible for any $n\ge1$. 

Next, assume that $F_{k-1}$ is $n$-fold constructible. 
Let $\alpha\in F_k$, then $\alpha$ is a root of a minimal polynomial $p$ with coefficients in $F_{k-1}$, and its degree divides $[F_k : F_{k-1}]$. 
If $\alpha$ is real, then it may be constructed by $n$-fold origami (Theorem \ref{thm_a}). If not, then its complex conjugate $\bar{\alpha}$ is also a root of $p$. The real and imaginary parts of $\alpha$, $\Re(\alpha)=(\alpha +\bar{\alpha})/2$ and $\Im(\alpha)=(\alpha -\bar{\alpha})/2$, respectively, are in $F_k$ and therefore they are real roots of minimal polynomials $p_\Re$ and $p_\Im$ with coefficients in $F_{k-1}$. Again, the degrees of both $p_\Re$ and $p_\Im$ divide $[F_k : F_{k-1}]$ and hence $\Re(\alpha)$ and $\Im(\alpha)$ are $n$-fold origami constructible.
\end{proof}

Using the above lemmas, we finally  obtain a sufficient condition for the constructability of the regular $m$-gon.

\begin{thm}
The regular $m$-gon is $n$-fold constructible if the largest prime factor $p$ of $\phi(m)$ satisfies $p\le n+2$.
\label{thm_fin}
\end{thm}

\begin{proof}

Let $\phi(m)= p_1p_2\cdots p_k$, where each $p_i$ is a prime and $p_i\le n+2$, and $\xi_m$ be a primitive $m$th root of unity. The Galois group $\Gamma$ of the extension $\mathbb{Q}(\xi_m):\mathbb{Q}$ is abelian and has order $\phi(m)$
\citep{Stewart2015}. Therefore, it has a series of normal subgroups $1=\Gamma_0\subseteq \Gamma_1 \subseteq \cdots \subseteq \Gamma_r = \Gamma$ 
where each factor $\Gamma_{j+1}/|\Gamma_j$ is abelian and has order $p_i$ for some $1\le i \le k$. By the Galois correspondence, there is a field tower $\mathbb{Q}(\xi_m)=K_0 \supseteq K_1 \supseteq \cdots \supseteq K_r = \mathbb{Q}$ such that $[K_j: K_{j+1}]=p_i$. 
Thus, by Lemma \ref{alpha}, $\xi_m$ is $n$-fold constructible, and by Lemma \ref{lemmaroot}, the $m$-gon is $n$-fold constructible. 
\end{proof}

\begin{exmp}
The totient of 199 is $\phi(199)=2\cdot 3^2\cdot11$. Therefore, the regular 199-gon may be constructed by 9-fold origami.   
\end{exmp}

\section{Final comments}

\citet{Gleason1988} noted that any regular $m$-gon may be constructed if, in addition to straight edge and compass, a tool to $p$-sect any angle is available for every prime factor $p$ of $\phi(m)$. The above results match his conclusion: if  $n$-fold origami can $p$-sect any angle for every prime factor $p$ of $\phi(m)$, then, by Lemma \ref{thm_angle}, the largest prime factor is $p_\text{max}\le n+2$. By Theorem \ref{thm_fin}, the $m$-gon can be construted.

\end{document}